\documentclass[12pt]{amsart}
 \usepackage[dvips]{epsfig}
 \usepackage{comment}
 \usepackage{enumerate}
 \usepackage{amsgen, amstext,amsbsy,amsopn, amsthm, amsfonts,amssymb,amscd,amsmat
 h,euscript,enumerate,url,verbatim,calc,xypic, mathtools}

 \usepackage{latexsym}
 \usepackage{graphics}
 \usepackage{color}

\newcommand{\Pic}{\rm Pic\,}

\newcommand{\proset}{\,\mathrel{\lower 4pt\hbox{$\scriptscriptstyle/$}
\mkern -14mu\subseteq }\,} 
 
 \newtheorem{theorem}{Theorem}[section]
  \newtheorem{corollary}[theorem]{Corollary}
 \newtheorem{lemma}[theorem]{Lemma}

\newtheorem{remark}[theorem]{Remark}

 \newtheorem{example}[theorem]{Example}

\numberwithin{equation}{section}

\def\Br{\operatorname{Br}}
\usepackage{amsmath}
 \makeatother 
 
\begin{document}
\date{\today}

\title{The Brauer group of a valuation ring}
 \author{Vivek Sadhu} 
 
 \address{Department of Mathematics, Indian Institute of Science Education and Research, Bhopal, Bhopal Bypass Road, Bhauri, Bhopal-462066, Madhya Pradesh, India}
 \email{vsadhu@iiserb.ac.in, viveksadhu@gmail.com}
 \thanks{Author was supported by SERB-DST MATRICS grant MTR/2018/000283}

 \begin{abstract}
  Let $V$ be a valuation ring and $K$ be its field of fraction. We show that the canonical map $\Br(V) \to \Br(K)$ is injective.
 \end{abstract}

\maketitle
 
\section{Introduction}
In recent years, it has been observed by various authors that a valuation ring (not necessarily Noetherian) behaves almost like regular ring in many situations. One can see \cite{KM} for some $K$-theoretic evidence. In this article, we give some evidence for Brauer groups. 
The Brauer group $\Br(R)$ of a commutative ring $R$ consists of Morita equivalence classes of Azumaya algebras over $R$ (see \cite{AG}, \cite{Milne}). The group operation on $\Br(R)$ is $\otimes_{R}.$ An element of $\Br(R)$ is represented by a class $[A],$ where $A$ is an Azumaya algebra over $R.$

A classical result of Auslander and Goldman says that for a regular noetherian domain $R$ with the field of fraction $K,$ the canonical map $\Br(R) \to \Br(K)$ sending $[A]$ to $[A\otimes_{R} K]$ is injective (see Theorem 7.2 of \cite{AG}). We observe that the same is true for valuation rings (see below Theorem \ref{ger inj}). A subring $V$ of a field $K$ is said to be valuation ring if for each nonzero element $a\in K,$ either $a\in V$ or $a^{-1}\in V.$ Clearly, $K$ itself a valuation ring and we  call it a trivial valuation ring of $K.$ If $V\neq K$ then we say that $V$ is a nontrivial valuation ring of $K.$ Every valuation ring is a local normal domain. If $\mathfrak{p}$ is a prime ideal of a valuation ring $V$ then the quotient $V/\mathfrak{p}$ and the localization $V_{\mathfrak{p}}$ both are valuation rings. The rank of a valuation ring is equal to its krull dimension. Here is our main result.
\begin{theorem}\label{ger inj}
Let $V$ be a valuation ring with the field of fraction $K.$ Then the canonical map $\Br(V) \to \Br(K)$ is injective.  
 \end{theorem}

 {\bf Acknowledgement:} The author would like to thank Prof. Jean-Louis Colliot Thelene for his valuable comments on the first draft of this article.
 \section{One dimensional case}\label{dim 1 case}
 In this section, we prove Theorem \ref{ger inj} in the case when $V$ is a one dimensional valuation ring with the field of fraction $K.$ Throughout this section, let $V$ be a one dimensional valuation ring with the field of fraction $K$ and $f: {\rm {\rm Spec}}(K) \to {\rm {\rm Spec}}(V)$ be the map of schemes induced by the inclusion $V \hookrightarrow K.$ Given a ring $R,$ $R^{\times}$ always denotes the group of units.
 
 The Leray spectral sequence of the map $f: {\rm {\rm Spec}}(K) \to {\rm {\rm Spec}}(V)$ gives  \tiny
  \begin{equation}\label{5 term}
  \begin{split}
   0\to H^{1}_{et}({\rm Spec}(V), f_{*}\mathbb{G}_{m, K}) \to H^{1}_{et}({\rm Spec}(K), \mathbb{G}_{m, K}) \to H^{0}_{et}({\rm Spec}(V), R^{1}f_{*}\mathbb{G}_{m, K})  \\ \to H^{2}_{et}({\rm Spec}(V), f_{*}\mathbb{G}_{m, K}) \to H^{2}_{et}({\rm Spec}(K), \mathbb{G}_{m, K}), \end{split}
  \end{equation}\normalsize
where $\mathbb{G}_{m}$ denotes the \'etale unit sheaf.
  We know $H^{1}_{et}({\rm Spec}(K), \mathbb{G}_{m, K})\cong\Pic({\it K})=0$ and $H^{2}_{et}({\rm Spec}(K), \mathbb{G}_{m, K})\cong \Br(K).$ The stalk of $R^1f_{*}\mathbb{G}_{m, K}$ at a geometric point $\bar{v}$ is $H^{1}_{et}({\rm Spec}(K_{\bar{v}}), \mathbb{G}_{m, K_{\bar{v}}})\cong \Pic({\it K_{\bar{v}}})$ where $K_{\bar{v}}= K\otimes_{V} \mathcal{O}_{V, \bar{v}}.$ Since $\dim (V)=1,$ $\mathcal{O}_{V, \bar{v}}$ is either $K^{sep}$ or $V^{sh},$ where $K^{sep}$ is any separable closure of $K.$ Moreover, $V^{sh}$ is a normal domain of dimension one (see \cite[Tag 07QL]{sp}). Thus, $K_{\bar{v}}$ is either $K^{sep}$ or the field of fraction of $V^{sh}.$ Therefore, $R^1f_{*}\mathbb{G}_{m, K}=0.$ By (\ref{5 term}) we get
  \begin{equation}\label{after calculation}
   H^{1}_{et}({\rm Spec}(V), f_{*}\mathbb{G}_{m, K})=0 ~{\rm and}~    0 \to H^{2}_{et}({\rm Spec}(V), f_{*}\mathbb{G}_{m, K}) \to \Br(K).
  \end{equation}
  Now, the long exact sequence associated to the exact sequence 
                                                                                                                                                                                                                                                $$1 \to \mathbb{G}_{m, V} \to f_{*}\mathbb{G}_{m, K} \to f_{*}\mathbb{G}_{m, K}/\mathbb{G}_{m, V} \to 1$$ of \'etale sheaves on ${\rm {\rm Spec}}(V)$ implies (using (\ref{after calculation}) and an wellknown fact that the Picard group of a local ring is trivial) 
              \begin{align}\label{seq 0}
               H^{0}_{et}({\rm Spec}(V), f_{*}\mathbb{G}_{m, K}/\mathbb{G}_{m, V})\cong K^{\times}/V^{\times}
              \end{align}
           and                                                                                                                                                                                                                        \begin{align}\label{seq2}
    0\to H^{1}_{et}({\rm Spec}(V), f_{*}\mathbb{G}_{m, K}/\mathbb{G}_{m, V})\to H^{2}_{et}({\rm Spec}(V), \mathbb{G}_{m, V})\to H^{2}_{et}({\rm Spec}(V), f_{*}\mathbb{G}_{m, K}).                                                                                                                                                                                                                                                                                                                                                                                                  
                                                                                                                                                                                                                                    \end{align}
        Our first goal is to prove the following result:
        \begin{theorem}\label{first coho of quotiont sheaf vanish}
         Let $V$ be a one dimensional valuation ring with the field of fraction $K.$ Let $f: {\rm {\rm Spec}}(K) \to {\rm {\rm Spec}}(V)$ be the map of schemes induced by the inclusion $V \hookrightarrow K.$ Then $H^{1}_{et}({\rm Spec}(V), f_{*}\mathbb{G}_{m, K}/\mathbb{G}_{m, V})=0.$
        \end{theorem}

         We need some preparations to prove Theorem \ref{first coho of quotiont sheaf vanish}. Let us begin with a definition.
        
        Let $S$ be any scheme. The support of an abelian sheaf $\mathcal{F}$ on $S_{et}$ is the set of points $s\in S$ such that $\mathcal{F}_{\bar{s}}\neq 0$ for any geometric point $\bar{s}$ lying over $s$. It is denoted by ${\rm Supp}(\mathcal{F}).$
        
        \begin{lemma}\label{support}
         Suppose $V$, $K$ and $f$ are as in Theorem \ref{first coho of quotiont sheaf vanish}. Let $\mathfrak{m}$ denote the unique maximal ideal of $V$. Then ${\rm Supp}(f_{*}\mathbb{G}_{m, K}/\mathbb{G}_{m, V})$ is ${\rm Spec}(V/\mathfrak{m}).$
        \end{lemma}
        
        \begin{proof}
         Write $\mathcal{I}$ for the \'etale sheaf $f_{*}\mathbb{G}_{m, K}/\mathbb{G}_{m, V}.$ The stalk of $\mathcal{I}$ at a geometric point $\bar{v}$ is $\mathcal{I}_{\bar{v}}=  (K\otimes_{V} \mathcal{O}_{V, \bar{v}})^{\times}/\mathcal{O}_{V, \bar{v}}^{\times}.$ Since $\dim{V}=1,$ ${\rm Spec}(V)=\{\xi, \mathfrak{m}\}$ where $\xi$ is the generic point of ${\rm Spec}(V).$  Then $\mathcal{I}_{\bar{\xi}}=0$ and $\mathcal{I}_{\bar{\mathfrak{m}}}= K(V^{sh})^{\times}/(V^{sh})^{\times},$ where $K(V^{sh})$ denotes the field of fraction of $V^{sh}.$ Note that $K(V^{sh})^{\times}\neq (V^{sh})^{\times}$ because $V^{sh}$ is a one dimensional local domain. Thus, the support of $\mathcal{I}$ is ${\rm Spec}(V/\mathfrak{m}).$
        \end{proof}
 Let $k$ denote the residue field $V/\mathfrak{m}.$ We have a closed immersion $i: {\rm Spec}(k) \to {\rm Spec}(V).$ Write $\mathcal{I}$ for  $f_{*}\mathbb{G}_{m, K}/\mathbb{G}_{m, V}$ as before. By Lemma \ref{support}, ${\rm Supp}(\mathcal{I})= {\rm Spec}(k).$ Then Proposition 58.46.4 of \cite[Tag 04E1]{sp} implies that $i_{*}(i^{-1}\mathcal{I})\cong \mathcal{I}.$ Here $i^{-1}\mathcal{I}$ is a sheaf on $({\rm Spec}(k))_{et}.$ Set $\mathcal{I}^{'}= i^{-1}\mathcal{I}.$ Since $i_{*}$ is an exact functor, $R^{q}i_{*}\mathcal{I}^{'}=0$ for all $q>0.$ The spectral sequence $$H_{et}^{p}({\rm Spec}(V), R^{q}i_{*}\mathcal{I}^{'})\Rightarrow H_{et}^{p+q}({\rm Spec}(k), \mathcal{I}^{'})$$
 gives the following:
 \begin{lemma}\label{cohomology at point}
  With the above notations, there is an isomorphism of groups $$H_{et}^{j}({\rm Spec}(V), \mathcal{I})\cong H_{et}^{j}({\rm Spec}(k), i^{-1}\mathcal{I})~ {\rm for~ all}~ j\geq 0.$$
 \end{lemma}

  The absolute Galois group of $k$ is ${\rm Gal}(k^{sep}/k)$ and is denoted by $G_{k}.$ Here $k^{sep}$ is the separable closure of $k$ in the algebraically closed field $\bar{k}.$ Note that $(i^{-1}\mathcal{I})_{\bar{\mathfrak{m}}}=\mathcal{I}_{\bar{\mathfrak{m}}}= K(V^{sh})^{\times}/(V^{sh})^{\times}.$ Moreover, $K(V^{sh})^{\times}/(V^{sh})^{\times}$ is a $G_{k}$-module (see \cite[Tag 03QW]{sp}). So, we can consider its Galois cohomology group $H^{j}(G_{k}, K(V^{sh})^{\times}/(V^{sh})^{\times})$ for $j\geq 0.$ Some details related to Galois cohomology can be found in \cite{GS} and \cite{serre}. Given a $G$-module $M,$ the invariant subgroup $M^{G}=\{x\in M|~ g.x=x ~{\rm for~ all}~ g\in G\}.$ Here $G$ is any group. Now, Lemma 58.58.2 of \cite[Tag 03QQ]{sp} implies the following:
  
  \begin{lemma}\label{cohomology of pt vs galois} With the above notations, we have 
  \begin{enumerate}
   \item $H_{et}^{0}({\rm Spec}(k), i^{-1}\mathcal{I})= (K(V^{sh})^{\times}/(V^{sh})^{\times})^{G_{k}},$
   \item $H_{et}^{j}({\rm Spec}(k), i^{-1}\mathcal{I})= H^{j}(G_{k}, K(V^{sh})^{\times}/(V^{sh})^{\times}) ~{\rm for}~ j>0.$
  \end{enumerate}
\end{lemma}

By combining the above two lemmas, we obtain
\begin{align}\label{global section}
 H_{et}^{0}({\rm Spec}(V), \mathcal{I})\cong H_{et}^{0}({\rm Spec}(k), i^{-1}\mathcal{I})=(K(V^{sh})^{\times}/(V^{sh})^{\times})^{G_{k}},
\end{align}

\begin{align}\label{higher cohomology}
 H_{et}^{j}({\rm Spec}(V), \mathcal{I})\cong H_{et}^{j}({\rm Spec}(k), i^{-1}\mathcal{I})= H^{j}(G_{k}, K(V^{sh})^{\times}/(V^{sh})^{\times}) ~{\rm for}~ j>0.
\end{align}

We are now ready to prove Theorem \ref{first coho of quotiont sheaf vanish}.

{\it Proof of Theorem \ref{first coho of quotiont sheaf vanish}:} The strict henselization $V^{sh}$ of $V$ is also a one dimensional valuation ring (see Lemma 15.118.5 of \cite[Tag 0ASF]{sp}). Recall $K(V^{sh})$ is the field of fraction of $V^{sh}.$ Then, we have $K(V^{sh})^{\times}/(V^{sh})^{\times}\cong K^{\times}/V^{\times}\cong (K(V^{sh})^{\times}/(V^{sh})^{\times})^{G_{k}},$ where the first isomorphism by Lemma 15.118.5 of \cite[Tag 0ASF]{sp} and the second isomorphism follows from (\ref{seq 0}) and (\ref{global section}). Hence, $G_{k}$ acts trivially on $K(V^{sh})^{\times}/(V^{sh})^{\times}.$ By (\ref{higher cohomology}), we also have $H_{et}^{1}({\rm Spec}(V), f_{*}\mathbb{G}_{m, K}/\mathbb{G}_{m, V})\cong H^{1}(G_{k}, K(V^{sh})^{\times}/(V^{sh})^{\times}).$ So, it is enough to show that $ H^{1}(G_{k}, K(V^{sh})^{\times}/(V^{sh})^{\times})=0.$ Since $\dim{V^{sh}}=1,$ the value group $K(V^{sh})^{\times}/(V^{sh})^{\times}$ is isomorphic to a non-trivial additive subgroup of $\mathbb{R}$ (see chapter VI, section 4.5, Proposition 8 of \cite{Bou}). In fact, $K(V^{sh})^{\times}/(V^{sh})^{\times}$ is a torsion free $\mathbb{Z}[G_{k}]$-module. As $G_{k}$ is a profinite group (see Proposition 4.1.3 of \cite{GS}), $G_{k}=\varprojlim {\rm Gal}(L/k),$ where the limit over all  finite Galois subextensions $L/k$ of $k^{sep}/k.$ Thus, $$H^{1}(G_{k}, K(V^{sh})^{\times}/(V^{sh})^{\times})\cong \varinjlim H^{1}({\rm Gal}(L/k), K(V^{sh})^{\times}/(V^{sh})^{\times}).$$ There are natural surjections $G_{k} \to {\rm Gal}(L/k)$ for all  finite Galois subextensions $L/k$ of $k^{sep}/k$ (see Corollary 4.1.4 of \cite{GS}). Hence, all ${\rm Gal}(L/k)$ acts trivially on $K(V^{sh})^{\times}/(V^{sh})^{\times}.$  In this situation, we know (see Exercise 6.1.5 of \cite{weihom}) $$H^{1}({\rm Gal}(L/k), K(V^{sh})^{\times}/(V^{sh})^{\times})\cong {\rm Hom}_{{\rm Groups}}({\rm Gal}(L/k), K(V^{sh})^{\times}/(V^{sh})^{\times}),$$ where ${\rm Hom}_{{\rm Groups}}(-, -)$ denotes the hom set in the category of groups.  Since ${\rm Gal}(L/k)$ is a finite group, ${\rm Hom}_{{\rm Groups}}({\rm Gal}(L/k), K(V^{sh})^{\times}/(V^{sh})^{\times})=0.$ Hence the result. \qed
    \begin{corollary}\label{one dim situation}
     Let $V$ be a one dimensional valuation ring with the field of fraction $K.$ Then the following are true:
     \begin{enumerate}
      
      \item $H^{2}_{et}({\rm Spec}(V), \mathbb{G}_{m, V})$ is a torsion group.
      \item $\Br(V)\cong H^{2}_{et}({\rm Spec}(V), \mathbb{G}_{m, V}).$
      \item the canonical map $\Br(V) \to \Br(K)$ is injective.
     \end{enumerate}
\end{corollary}
    \begin{proof}
    (1) The Brauer group of a field always a torsion group (see Corollary 2.8.5 of \cite{GS}). Thus, $H^{2}_{et}({\rm Spec}(V), f_{*}\mathbb{G}_{m, K})$ is a torsion group by (\ref{after calculation}). Hence the result by Theorem \ref{first coho of quotiont sheaf vanish} and (\ref{seq2}).
     
     (2) By a theorem of Gabber (see \cite{Djong}), $\Br(R)\cong (H^{2}_{et}({\rm Spec}(R), \mathbb{G}_{m, R}))_{tor} $ for any commutative ring $R.$ We get the assertion by (1).
     
     (3) This is now immediate from Theorem \ref{first coho of quotiont sheaf vanish}, (\ref{after calculation}) and (\ref{seq2}).
    \end{proof}

                                                                                                                                                                                                                                    
    \section{general case} \label{main res}
    In this section, we prove the general case of Theorem \ref{ger inj}. We also discuss some consequences of Theorem \ref{ger inj}.
    
  {\it Proof of Theorem \ref{ger inj}:}                                                                                                                                                                                                                                          We may assume that $V\neq K.$ Note that $K=\bigcup K_{i},$ where $K_{i}$ is a finitely generated subfield of $K.$ Then $V$ can be wriiten as filtered limit of $V_{i},$ where $V_{i}=V\cap K_{i}$ is a valuation ring of $K_{i}$ with finite rank (see Lemma 2.22 of \cite{Bhat-Matthew}). Since $\Br$ commutes with filtered limit, we  may also assume that $V$ is a valuation ring of finite rank $d>0.$  We use induction on $d.$ If $d=1$ then the result is true by Corollary \ref{one dim situation}(3). Suppose that the result is true for all valuation rings of rank less than $d.$ Let $V$ be a valuation ring of rank $d$ with the field of fraction $K.$ Let $\mathfrak{p}$ be a nonzero, non-maximal prime ideal of $V.$ Note that ${\rm {\rm Spec}}(V_{\mathfrak{p}})$ is open in ${\rm {\rm Spec}}(V).$ Moreover, $V_{\mathfrak{p}}$ and $V/\mathfrak{p}$ both are valuation rings with the field of fractions $K$ and $k(\mathfrak{p})=V_{\mathfrak{p}}/\mathfrak{p}V_{\mathfrak{p}}.$ We also have $1\leq \dim(V_{\mathfrak{p}})< d$ and $1\leq \dim(V/\mathfrak{p})< d.$ Since $V$ is a valuation ring, the following square (see Proposition 2.8 of \cite{Bhat-Matthew})
                                                                                                                                                                                                                                             $$\begin{CD}
    V @>>> V_{\mathfrak{p}} \\
    @VVV   @VVV \\
    V/\mathfrak{p} @>>> k(\mathfrak{p})
                                                                                                                                                                                                                            \end{CD}$$
is Milnor. By Corollary 2.5 of \cite{KO}, we get an exact sequence $$ 0 \to \Br(V) \to \Br(V_{\mathfrak{p}}) \times \Br(V/\mathfrak{p}) \to \Br(k(\mathfrak{p})).$$ Now consider the following commutative diagram 
$$\begin{CD}
 0 @>>> \Br(V) @>>> \Br(V_{\mathfrak{p}}) \times \Br(V/\mathfrak{p}) \\
 @.  @VVV   @VVV   \\
  @. \Br(K) @>>> \Br(K) \times \Br(k(\mathfrak{p})),
\end{CD}$$where the second vertical map is injective by the induction hypothesis. Hence the result.\qed

\begin{corollary}\label{inj for all prime}
 Let $V$ be a valuation ring with the field of fraction $K.$ Let $\mathfrak{p}, \mathfrak{q} \in {\rm Spec}(V).$ Then either $\Br(V_{\mathfrak{p}})\hookrightarrow \Br(V_{\mathfrak{q}})$ or $ \Br(V_{\mathfrak{q}})\hookrightarrow \Br(V_{\mathfrak{p}}).$ In particular, $\Br(V) \hookrightarrow \Br(V_{\mathfrak{p}})$ for all $\mathfrak{p}\in {\rm Spec}(V).$
\end{corollary}
\begin{proof}
 Since $V$ is a valuation ring, either $\mathfrak{p}\subseteq \mathfrak{q}$ or $\mathfrak{q}\subseteq \mathfrak{p}.$ Note that the both $V_{\mathfrak{p}}$ and $V_{\mathfrak{q}}$ are valuation rings with the field of fraction $K.$ The assertion now clear from Theorem \ref{ger inj}. For the second part, consider $\mathfrak{p}\subseteq \mathfrak{m},$ where $\mathfrak{m}$ is the unique maximal ideal.
\end{proof}
\begin{remark}\label{chain of grps}{\rm
 For a valuation ring $V,$ we have exactly one chain of prime ideals, say 
 $$(0) \subseteq \mathfrak{p}\subseteq \dots \subseteq \mathfrak{q} \subseteq \mathfrak{m}.$$ By Corollary \ref{inj for all prime}, we get $$\Br(V)\hookrightarrow Br(V_{\mathfrak{q}})\hookrightarrow \dots \hookrightarrow \Br(V_{\mathfrak{p}})\hookrightarrow \Br(K).$$ In particular, $\Br(V)=\bigcap_{\mathfrak{p}\in {\rm Spec}(V)} \Br(V_{\mathfrak{p}}).$}
\end{remark}

\begin{remark}{\rm
 Kestutis $\check{{\rm C}}$esnavi$\check{{\rm c}}$ius inform me in \cite{KC} that a stronger version of Theorem \ref{ger inj} is also true over valuation rings. More explicitly, for a valuation ring $V$ with the field of fraction $K,$ the canonical map $H_{et}^{2}({\rm Spec}(V), \mathbb{G}_{m, V}) \to H_{et}^{2}({\rm Spec}(K), \mathbb{G}_{m, K})$ is injective.}
\end{remark}

\end{document}